\documentclass[reqno]{amsart}
\usepackage{amsmath, amsthm, amssymb, amstext}

\usepackage{hyperref,xcolor}
\hypersetup{
 pdfborder={0 0 0},
 colorlinks,
}
\usepackage{enumitem}
\newtheorem{theorem}{Theorem}
\newtheorem{remark}[theorem]{Remark}
\newtheorem{lemma}[theorem]{Lemma}
\newtheorem{proposition}[theorem]{Proposition}

\DeclareMathOperator*{\divergenz}{div}              %

\newcommand{\Lp}[1]{L^{#1}(\Omega)}

\newcommand{\Wp}[1]{W^{1,#1}(\Omega)}

\newcommand{\Wpzero}[1]{W^{1,#1}_0(\Omega)}

\newcommand{\eps}{\varepsilon}

\newcommand{\into}{\int_{\Omega}}

\newcommand{\close}{\overline{\Omega}}

\newcommand{\cprime}{$'$}

\renewcommand{\l}{\left}
\renewcommand{\r}{\right}

\numberwithin{theorem}{section}
\numberwithin{equation}{section}

\DeclareMathOperator*{\esssup}{ess\,sup}
\DeclareMathOperator*{\loc}{loc}
\newcommand*\diff{\mathop{}\!\mathrm{d}}

\newcommand{\norm}[1]{\left \Vert #1\right\Vert}

\newcommand{\R}{{\mathbb R}}

\newcommand{\N}{{\mathbb N}}

\newcommand{\Assg}[1]{\textup{(g)}}

\title[Boundedness and H\"older continuity of weak solutions]
{The boundedness and H\"older continuity of weak solutions to elliptic equations involving variable exponents and critical growth}

\author[K. Ho]{Ky Ho}
\address[K. Ho]{Institute of Applied Mathematics, University of Economics Ho Chi Minh City, 59C, Nguyen Dinh Chieu Street, District 3, Ho Chi Minh City, Viet Nam}
\email{kyhn@ueh.edu.vn}

\author[Y.-H. Kim]{Yun-Ho Kim}
\address[Y.-H. Kim]{Department of Mathematics Education, Sangmyung University, Seoul, 03016, Korea}
\email{kyh1213@smu.ac.kr}

\author[P. Winkert]{Patrick Winkert}
\address[P. Winkert]{Technische Universit\"{a}t Berlin, Institut f\"{u}r Mathematik, Stra\ss e des 17.\,Juni 136, 10623 Berlin, Germany}
\email{winkert@math.tu-berlin.de}

\author[C. Zhang]{Chao Zhang}
\address[C. Zhang]{School of Mathematics and Institute for Advanced Study in Mathematics, Harbin Institute of Technology, Harbin 150001, PR China}
\email{czhangmath@hit.edu.cn}


\begin{document}

\begin{abstract}
	In this paper we prove the boundedness and H\"older continuity of quasilinear elliptic problems involving variable exponents for a homogeneous Dirichlet and a nonhomogeneous Neumann boundary condition, respectively. The novelty of our work is the fact that we allow critical growth even on the boundary and so we close the gap in the papers of Fan-Zhao [Nonlinear Anal. {\bf 36} (1999), no. 3, 295--318.] and Winkert-Zacher [Discrete Contin. Dyn. Syst. Ser. S {\bf 5} (2012), no. 4, 865--878.] in which the critical cases are excluded. Our approach is based on a modified version of De Giorgi's iteration technique along with the localization method. As a consequence of our results, the $C^{1,\alpha}$-regularity follows immediately.
\end{abstract}
	
\maketitle

\section{Introduction}
Let $\Omega\subset \R^N$ ($N\geq 2$) be a bounded domain with Lipschitz boundary $\Gamma:=\partial\Omega$ and let $p \in C(\close)$ be such that $1<p^-:=\min_{\close}p$. In this paper, we investigate the boundedness and H\"older continuity of weak solutions to quasilinear elliptic problems defined in variable exponent Sobolev spaces involving critical growth of the general form
\begin{equation}\label{D}
	\begin{aligned}
		-\divergenz \mathcal{A}(x,u,\nabla u)& =\mathcal{B}(x,u,\nabla u)\quad && \text{in } \Omega,\\
		u & = 0 &&\text{on } \Gamma,
	\end{aligned}
\end{equation}
and
\begin{equation}\label{N}
	\begin{aligned}
		-\divergenz \mathcal{A}(x,u,\nabla u)& =\mathcal{B}(x,u,\nabla u)\quad && \text{in } \Omega,\\
		\mathcal{A}(x,u,\nabla u)\cdot \nu & = \mathcal{C}(x,u) &&\text{on } \Gamma,
	\end{aligned}
\end{equation}
where $\nu(x)$ denotes the outer unit normal of $\Omega$ at $x\in \Gamma$ and the functions $\mathcal{A}\colon\Omega\times\R\times\R^N\to \R^N$, $\mathcal{B}\colon\Omega \times \R\times \R^N\to \R$ as well as $\mathcal{C}\colon \Gamma \times \R \to\R$ are Carath\'eodory functions
which satisfy suitable $p(\cdot)$-structure conditions, see H(D), H(N) and H(A) in Sections \ref{section_3}, \ref{section_4} and \ref{section_5}, respectively, for the precise assumptions.

An important special case of \eqref{D} and \eqref{N} which is included in our setting is given by
\begin{align*}
	-\Delta_{p(\cdot)}u&=\mathcal{B}(x,u,\nabla u) \quad\text{in }\Omega,\qquad u=0 \quad \text{on }\Gamma,\\
	-\Delta_{p(\cdot)}u&=\mathcal{B}(x,u,\nabla u)\quad\text{in }\Omega,\qquad |\nabla u|^{p(x)-2} \nabla u \cdot \nu = \mathcal{C}(x,u) \quad \text{on }\Gamma,
\end{align*}
where $\Delta_{p(\cdot)}$ denotes the so-called $p(\cdot)$-Laplace differential operator which is given by
\begin{align}\label{p(x)-Laplacian}
	\Delta_{p(\cdot)}u=\divergenz(|\nabla u|^{p(\cdot)-2}\nabla u) \quad \text{for } u \in \Wpzero{p(\cdot)} \text{ or }\Wp{p(\cdot)}.
\end{align}
Note that $\Delta_{p(\cdot)}$ reduces to the well-known $p$-Laplacian $\Delta_p$ when $p(x) \equiv p$ is a constant.

Nonlinear equations of the form considered in \eqref{D} and \eqref{N} involving variable exponents in the structure conditions are usually called equations with nonstandard growth. Such equations are of great interest and appear in the mathematical modelling of certain physical phenomena, for example in fluid dynamics (flows of electro-rheological fluids or fluids with temperature-dependent viscosity), in nonlinear viscoelasticity, in image processing and in processes of filtration through porous media, see for example, Acerbi-Mingione-Seregin \cite{Acerbi-Mingione-Seregin-2004}, Antontsev-D{\'{\i}}az-Shmarev \cite{Antontsev-Diaz-Shmarev-2002}, Antontsev-Rodrigues \cite{Antontsev-Rodrigues-2006}, Chen-Levine-Rao \cite{Chen-Levine-Rao-2006}, Diening \cite{Diening-2002}, Rajagopal-R$\mathring{\text{u}}$\v{z}i\v{c}ka \cite{Rajagopal-Ruzicka-2001}, R$\mathring{\text{u}}$\v{z}i\v{c}ka \cite{Ruzicka-2000} and Zhikov \cite{Zhikov-1997, Zhikov-1997-b} and the references therein.

The $p(\cdot)$-Laplace differential operator given in \eqref{p(x)-Laplacian} is related to the energy functional
\begin{align}\label{energy-functional}
	u \mapsto \into |\nabla u|^{p(x)}\diff x,
\end{align}
which was first used to illustrate models for strongly anisotropic materials, see Zhikov \cite{Zhikov-1986, Zhikov-1995,Zhikov-1997-b}, that is,  in a material made of different components, the exponent $p(\cdot)$ dictates the geometry of a composite that changes its hardening exponent according to the point. We refer to the paper of Acerbi-Mingione \cite{Acerbi-Mingione-2005} who developed gradient estimates and qualitative properties of minimizers of energies that have variable exponents as in \eqref{energy-functional}.

Global a priori bounds for classes of elliptic problems with different boundary condition are important results in the theory of differential equations and often guarantee further knowledge about the smoothness of solutions. We mention the famous works of Lieberman \cite{Lieberman-1988, Lieberman-1991} stating the $C^{1,\alpha}$-regularity up to the boundary of solutions to partial differential equations of second order both with nonhomogeneous Dirichlet- and nonlinear Neumann boundary condition. A crucial condition for the application of such results is the boundedness of the solutions. The corresponding $C^{1,\alpha}$-result for equations with $p(\cdot)$-structure has been published by Fan \cite{Fan-2007}.
In order to obtain such results, most of the papers make use of the so-called De Giorgi-Nash-Moser theory which provides iterative methods based on truncation techniques to get $L^\infty$-bounds for certain equations, see the works of De Giorgi \cite{De-Giorgi-1957}, Nash \cite{Nash-1958} and Moser \cite{Moser-1960}. The techniques developed in these papers provided powerful tools to prove local and global boundedness, the Harnack and the weak Harnack inequality and the H\"older continuity of weak solutions. For more information we refer to the monographs of Gilbarg-Trudinger \cite{Gilbarg-Trudinger-1983}, Lady{\v{z}}enskaja-Ural{\cprime}ceva \cite{Ladyzenskaja-Uralceva-1968},  Lady{\v{z}}enskaja-Solonnikov-Ural{\cprime}ceva \cite{Ladyzenskaja-Solonnikov-Uralceva-1968} and Lieberman \cite{Lieberman-1996}.

In the present paper we are going to prove the boundedness and H\"older continuity of weak solutions of the problems \eqref{D} and \eqref{N}. The novelty in our work is the fact that we allow critical growth on the perturbations $\mathcal{B}$ and $\mathcal{C}$, so in the domain and on the boundary. With our work we close the gap of the papers of Fan-Zhao \cite{Fan-Zhao-1999} for the Dirichlet and Winkert-Zacher \cite{Winkert-Zacher-2012,Winkert-Zacher-2015} for the Neumann problem for the subcritical case. Indeed, in these papers the critical cases are excluded and cannot be realized with their treatment. However, we will adopt some ideas and found a way to overcome the problem with the critical exponents by a modified version of De Giorgi's iteration technique along with an appropriate  application of the localization method.

Let us comment on some relevant known results on quasilinear elliptic problems with $p$- and $p(\cdot)$-structure. Concerning boundedness results for elliptic problems driven by the $p$- or ($p,q$)-Laplacian we refer to the papers of Garc\'{\i}a Azorero-Peral Alonso-Manfredi \cite{Garcia-Azorero-Peral-Alonso-Manfredi-2000} (for homogeneous Dirichlet condition), Motreanu-Motreanu-Papageorgiou \cite{Motreanu-Motreanu-Papageorgiou-2009} (for homogeneous Neumann condition) and Winkert \cite{Winkert-2010} (for nonhomogeneous Neumann condition). All these works are working with a De Giorgi-Nash-Moser iteration technique but only for the subcritical case. In the critical constant exponent case for Neumann problems we refer to the papers of Papageorgiou-R\u adulescu \cite{Papageorgiou-Radulescu-2016} (critical growth in the domain) and Marino-Winkert \cite{Marino-Winkert-2019} (critical growth even on the boundary). For elliptic problems with $p(\cdot)$-structure, the first work has been done by Fan-Zhao \cite{Fan-Zhao-1999} who proved the boundedness and H\"older continuity for weak solutions of problem \eqref{D} in the subcritical case. The corresponding cases with a homogeneous and nonhomogeneous Neumann boundary condition are proved in Gasi{\'n}ski-Papageorgiou \cite{Gasinski-Papageorgiou-2011} and Winkert-Zacher \cite{Winkert-Zacher-2012, Winkert-Zacher-2015}, respectively. All these works treated only problems with subcritical growth. A priori bounds for degenerate elliptic equations with variable growth can be found in Ho-Sim \cite{Ho-Sim-2014,Ho-Sim-2015} based on the methods of \cite{Fan-Zhao-1999} and \cite{Winkert-Zacher-2012}. Boundedness results for weighted elliptic equations with variable exponents and convection term has been developed in Ho-Sim \cite{Ho-Sim-2017} and a priori bounds for the fractional
$p(\cdot)$-Laplacian are recently studied by Ho-Kim \cite{Ho-Kim-2019}. To the best of our knowledge, there exists no work for boundedness results and H\"older continuity for problems \eqref{D} and \eqref{N} involving functions that have critical growth.

Finally, we mention papers which are very close to our topic dealing with certain types of a priori bounds for equations with $p$- or $p(\cdot)$-structure. We refer to Ding-Zhang-Zhou \cite{Ding-Zhang-Zhou-2020b, Ding-Zhang-Zhou-2020}, Garc\'{\i}a Azorero-Peral Alonso \cite{Garcia-Azorero-Peral-Alonso-1994}, Guedda-V\'eron \cite{Guedda-Veron-1989}, Marino-Winkert \cite{Marino-Winkert-2020, Marino-Winkert-2020b}, Pucci-Servadei \cite{Pucci-Servadei-2008}, Wang \cite{Wang-1991}, Winkert \cite{Winkert-2010b, Winkert-2014}, Winkert-Zacher \cite{Winkert-Zacher-2016}, Zhang-Zhou \cite{Zhang-Zhou-2012}, Zhang-Zhou-Xue \cite{Zhang-Zhou-Xue-2014} and the references therein.

The paper is organized as follows. In Section \ref{section_2} we recall the basic properties of variable Lebesgue and Sobolev spaces and state the main tools which we will need in later sections. The boundedness of weak solutions of problem \eqref{D} is then presented in Section \ref{section_3}, see Theorem \ref{Theo.D}. In Section \ref{section_4} we prove the boundedness of weak solutions for problem \eqref{N}, see Theorem \ref{Theo.N}, and Section \ref{section_5} is concerned with the H\"older continuity of weak solutions for both, \eqref{D} and \eqref{N}. These results are stated in Theorems \ref{C1,alpha.D} and \ref{C1,alpha.N}, respectively. Finally, we mention the $C^{1,\alpha}$-regularity of weak solutions due to Fan \cite{Fan-2007}.

\section{Preliminaries and Notations}\label{section_2}

In this section we recall some basic facts about Lebesgue and Sobolev spaces with variable exponents, see, for example Fan-Zhao \cite{Fan-Zhao-2001} and Kov{\'a}{\v{c}}ik-R{\'a}kosn{\'{\i}}k \cite{Kovacik-Rakosnik-1991}. To this end, let $\Omega$ be a bounded domain in $\mathbb{R}^N$ with Lipschitz boundary $\Gamma$. For $p\in C_+(\overline\Omega)$, where $C_+(\close)$ is given by
\begin{align*}
	C_+(\close)=\{h \in C(\close) \, : \, 1<h(x) \text{ for all }x\in \close\},
\end{align*}
we denote
\begin{align*}
	p^{-}:=\inf_{x\in \close} p(x) \quad \text{and} \quad  p^{+}:=\sup_{x\in \close} p(x).
\end{align*}
Moreover, let $M(\Omega)$ be the space of all measurable functions $u\colon \Omega\to\R$. Then, for a given $p \in C_+(\close)$, the variable exponent Lebesgue space $\Lp{p(\cdot)}$ is defined as
\begin{align*}
	\Lp{p(\cdot)}=\l\{u \in M(\Omega)\,:\, \into |u|^{p(x)}\diff x<\infty \r\}
\end{align*}
equipped with the Luxemburg norm given by
\begin{align*}
	\|u\|_{p(\cdot)} =\inf \l \{\lambda>0 \, : \, \into \l|\frac{u(x)}{\lambda}\r|^{p(x)}\diff x \leq 1 \r\}.
\end{align*}

The following propositions can be found in Diening-Harjulehto-H\"{a}st\"{o}-R$\mathring{\text{u}}$\v{z}i\v{c}ka \cite{Diening-Harjulehto-Hasto-Ruzicka-2011}.

\begin{proposition}
	Let $s,r\in C_+(\close)$ be such that $s(x)\le r(x)$ for all $x\in\close.$ Then, $L^{r(\cdot)}(\Omega)\subseteq  L^{s(\cdot)}(\Omega)$ and
	\begin{align*}
		\norm{u}_{s(\cdot)}\le
		2(1+|\Omega|)\norm{u}_{r(\cdot)}\quad \text{for all }u\in  L^{r(\cdot)}(\Omega).
	\end{align*}
	
\end{proposition}

\begin{proposition}
	The space $L^{p(\cdot) }(\Omega)$ is a separable and uniformly convex Banach space and its conjugate space is $L^{p'(\cdot)}(\Omega )$, where  $1/p(x)+1/p'(x)=1$ for all $x \in \close$. For any $u\in L^{p(\cdot)}(\Omega)$ and for any $v\in L^{p'(\cdot)}(\Omega)$, we have
	\begin{align*}
		\left|\int_\Omega uv\diff x\right|\leq\ 2 \|u\|_{p(\cdot)}\|v\|_{p'(\cdot)}.
	\end{align*}
\end{proposition}

The corresponding modular $\rho \colon L^{p(\cdot) }(\Omega)  \to \R$ is given by
\begin{align*}
	\rho (u) =\int_{\Omega }| u| ^{p(x) }\diff x\quad\text{for all }u\in L^{p(\cdot) }(\Omega)  .
\end{align*}
The following proposition shows the relation between the norm $\|\cdot\|_{p(\cdot)}$ and the modular $\rho$.
\begin{proposition}\label{norm-modular}
	For all $u\in L^{p(\cdot) }(\Omega) $ we have the following assertions:
	\begin{enumerate}
		\item[(i)]
			$\|u\|_{p(\cdot)}<1$ (resp. $=1$, $>1$)
			if and only if \  $\rho (u) <1$ (resp. $=1$, $>1$);
		\item[(ii)]
			if \  $\|u\|_{p(\cdot)}>1,$ then  $\|u\|^{p^{-}}_{p(\cdot)}\leq \rho (u) \leq \|u\|_{p(\cdot)}^{p^{+}}$;
		\item[(iii)]
			if \ $\|u\|_{p(\cdot)}<1,$ then $\|u\|_{p(\cdot)}^{p^{+}}\leq \rho(u) \leq \|u\|_{p(\cdot)}^{p^{-}}$.
	\end{enumerate}
	Consequently, the following inequality holds true
	\begin{align*}
		\|u\|_{p(\cdot)}^{p^{-}}-1\leq \rho (u) \leq \|u\|_{p(\cdot)}^{p^{+}}+1 \quad \text{for all }u\in L^{p(\cdot)}(\Omega ).
	\end{align*}
\end{proposition}

The corresponding variable exponent Sobolev spaces can be defined in the same way using the variable exponent Lebesgue spaces. For $p \in C_+(\close)$ the variable exponent Sobolev space $\Wp{p(\cdot)}$ is defined by
\begin{align*}
	\Wp{p(\cdot)}=\l\{ u \in \Lp{p(\cdot)} \,:\, |\nabla u| \in \Lp{p(\cdot)}\r\}
\end{align*}
endowed with the norm
\begin{align*}
	\|u\|_{1,p(\cdot)}=\|u\|_{p(\cdot)}+\|\nabla u\|_{p(\cdot)} \quad\text{for all } u \in \Wp{p(\cdot)},
\end{align*}
where $\|\nabla u\|_{p(\cdot)}= \|\,|\nabla u|\,\|_{p(\cdot)}$. On $W^{1,p(\cdot)}(\Omega)$, the norm $\|\cdot\|_{1,p(\cdot)}$ is  equivalent to
\begin{align*}
	\|u\|_{1}:=\inf\left\{\lambda >0 \ : \
	\int_\Omega
	\left[\l|\frac{\nabla u(x)}{\lambda}\r|^{p(x)}+\l|\frac{u(x)}{\lambda}\r|^{p(x)}\right]\;\diff x\le1\right\}
\end{align*}
with the relation
\begin{align}\label{equi.norms}
	\frac{1}{2}\|u\|_{1,p(\cdot)}\leq \|u\|_1\leq 2 \|u\|_{1,p(\cdot)}\quad \text{for all }u\in W^{1,p(\cdot)}(\Omega ).
\end{align}
Moreover, we define
\begin{align*}
	\Wpzero{p(\cdot)}= \overline{C^\infty_0(\Omega)}^{\|\cdot\|_{1,p(\cdot)}}.
\end{align*}
The spaces $\Wp{p(\cdot)}$ and $\Wpzero{p(\cdot)}$ are both separable and reflexive Banach spaces, in fact uniformly convex Banach spaces. In the space $\Wpzero{p(\cdot)}$, the Poincar\'e inequality holds, namely
\begin{align*}
	\|u\|_{p(\cdot)} \leq c_0 \|\nabla u\|_{p(\cdot)} \quad\text{for all } u \in \Wpzero{p(\cdot)}
\end{align*}
with some $c_0>0$. Therefore, we can consider on $\Wpzero{p(\cdot)}$ the equivalent norm
\begin{align*}
	\|u\|_{0}=\|\nabla u\|_{p(\cdot)} \quad\text{for all } u \in \Wpzero{p(\cdot)}.
\end{align*}
For $p \in C_+(\close)$ we introduce the critical Sobolev variable exponent $\hat{p}^*$ and the corresponding one $\hat{p}_*$ on the boundary defined by
\begin{align*}
	\hat{p}^*(x)=
	\begin{cases}
		\frac{Np(x)}{N-p(x)} & \text{if }p(x)<N,\\
		\infty & \text{if } N \leq p(x),
	\end{cases} \quad\text{for all }x\in\close
\end{align*}
and
\begin{align*}
	\hat{p}_*(x)=
	\begin{cases}
	\frac{(N-1)p(x)}{N-p(x)} & \text{if }p(x)<N,\\
	\infty  & \text{if } N \leq p(x),
	\end{cases} \quad\text{for all }x\in\Gamma.
\end{align*}
It is well known that $W^{1,p(\cdot)}(\Omega) \hookrightarrow L^{q(\cdot) }(\Omega )$ is compactly embedded for any $q\in C(\close)$ satisfying $1\leq  q(x)<\hat{p}^*(x)$ for all $x\in\close$, see, for example, Fan \cite{Fan-2010}, and that $W^{1,p(\cdot)}(\Omega) \hookrightarrow L^{r(\cdot) }(\Gamma )$ is compactly embedded for any $r\in C(\Gamma)$ satisfying $1\leq  r(x)<\hat{p}_*(x)$ for all $x\in\Gamma$, see, for example, Fan \cite[Corollary 2.4]{Fan-2008}.

In order to state critical embeddings, we need more regularity on the function $p$. To this end, we denote by $C^{0, \frac{1}{|\log t|}}(\close)$ the set of all functions $h\colon \close \to \R$ that are log-H\"older continuous, that is, there exists $C>0$ such that
\begin{align*}
	|h(x)-h(y)| \leq \frac{C}{|\log |x-y||}\quad\text{for all } x,y\in \close \text{ with } |x-y|<\frac{1}{2}.
\end{align*}

Now we can state the critical embedding from $\Wp{p(\cdot)}$ into $\Lp{\hat{p}^*(\cdot)}$, see Diening-Harjulehto-H\"{a}st\"{o}-R$\mathring{\text{u}}$\v{z}i\v{c}ka \cite[Corollary 8.3.2]{Diening-Harjulehto-Hasto-Ruzicka-2011} or Fan \cite[Proposition 2.2]{Fan-2010}.

\begin{proposition}\label{embedding_critical}
	Let $p\in C^{0, \frac{1}{|\log t|}}(\close) \cap C_+(\close)$ and let $q\in C(\close)$ be such that
	\begin{align*}
		1\leq  q(x)\leq \hat{p}^*(x) \quad \text{for all }x\in\close.
	\end{align*}
	Then, we have the continuous embedding
	\begin{center}
		$W^{1,p(\cdot)}(\Omega) \hookrightarrow L^{q(\cdot) }(\Omega ).$
	\end{center}
	In particular, we have
	\begin{align*}
		\|u\|_{q(\cdot)}\leq C\|u\|_{1,p(\cdot)} \quad \text{for all } u\in W^{1,p(\cdot)}(\Omega).
	\end{align*}
	If $q(x)< \hat{p}^*(x)$ for all $x\in\overline{\Omega}$, then the embedding above is compact.
\end{proposition}

The next theorem is an extension of the classical boundary trace embedding theorem for the variable exponent case which requires that $p$ belongs a subclass of  $C^{0, \frac{1}{|\log t|}}(\close)$. The case $p^+<N$  was obtained in Fan \cite[Theorem 2.1]{Fan-2008} and this restriction can be avoided as shown in the following proposition.

\begin{proposition}\label{embedding_critical_boundary}
	Suppose that $p\in C_+(\close)\cap W^{1,\gamma}(\Omega)$ for some $\gamma>N$. Let $s\in C(\close)$ be such that
	\begin{align*}
		1\leq  s(x)\leq \hat{p}_*(x) \quad \text{for all }x\in\close
	\end{align*}
	Then, it holds that
	\begin{align*}
		W^{1,p(\cdot)}(\Omega)\hookrightarrow L^{s(\cdot) }(\Gamma)
	\end{align*}
	is continuously embedded.
\end{proposition}

\begin{proof}
	Let $s\in C(\close)$ satisfying $1\leq  s(x)\leq \hat{p}_*(x)$ for all $x\in\close.$  We fix $r\geq s^+$ and define
	\begin{align*}
		\hat{r}:=\frac{Nr}{N-1+r}.
	\end{align*}
	Then $\hat{r}<N$ and
	\begin{align}\label{2.hat.r}
		r=\frac{(N-1)\hat{r}}{N-\hat{r}}.
	\end{align}
	For $x\in\close$, we define
	\begin{align}\label{2.tilde.p}
		\tilde{p}(x):=
		\begin{cases}
			p(x) & \text{if }p(x)<\hat{r}, \\
			\hat{r} & \text{if }p(x) \geq \hat{r}.
		\end{cases}
	\end{align}
	It is clear that $\tilde{p}\in C_+(\close)\cap W^{1,\gamma}(\Omega)$ and $\tilde{p}^+\leq \hat{r}<N$. If $p(x)<\hat{r}$, then by \eqref{2.tilde.p}, we have
	\begin{align*}
		s(x)\leq \hat{p}_*(x)=\frac{(N-1)\tilde{p}(x)}{N-\tilde{p}(x)}.
	\end{align*}
	In the case $p(x)\geq \hat{r}$, by the choice of $r$, \eqref{2.hat.r} and \eqref{2.tilde.p}, we obtain
	\begin{align*}
		s(x)\leq r=\frac{(N-1)\hat{r}}{N-\hat{r}}= \frac{(N-1)\tilde{p}(x)}{N-\tilde{p}(x)}.
	\end{align*}
	Thus we derive
	\begin{align*}
		s(x) \leq \frac{(N-1)\tilde{p}(x)}{N-\tilde{p}(x)}=:\tilde{p}_*(x)\quad \text{for all }x\in\close.
	\end{align*}
	Hence, from Fan \cite[Theorem 2.1]{Fan-2008} we get the continuous embeddings
	\begin{equation*}
	W^{1,p(\cdot)}(\Omega)\hookrightarrow L^{\tilde{p}_*(\cdot) }(\Gamma)\hookrightarrow L^{s(\cdot) }(\Gamma).
	\end{equation*}
	The proof is complete.
\end{proof}

\begin{remark}
	Note that for a bounded domain $\Omega\subset \R^N$ and $\gamma>N$ we have the following inclusions
	\begin{align}\label{inclusions}
		C^{0,1}(\close)\subset \Wp{\gamma}\subset C^{0,1-\frac{N}{\gamma}}(\close) \subset C^{0, \frac{1}{|\log t|}}(\close).
	\end{align}
	
\end{remark}
The following lemma concerning the geometric convergence of sequences of numbers will be needed for the De Giorgi iteration arguments below. It can be found in Ho-Sim \cite[Lemma 4.3]{Ho-Sim-2015}. The case $\mu_1=\mu_2$ is contained in Lady{\v{z}}enskaja-Solonnikov-Ural{\cprime}ceva \cite[Chapter II, Lemma 5.6]{Ladyzenskaja-Solonnikov-Uralceva-1968}, see also DiBenedetto \cite[Chapter I, Lemma 4.1]{DiBenedetto-1993}.
\begin{lemma}\label{leRecur}
	Let $\{Z_n\}, n=0,1,2,\ldots,$ be a sequence of positive numbers, satisfying the recursion inequality
	\begin{align*}
		Z_{n+1} \leq K b^n \left (Z_n^{1+\mu_1}+ Z_n^{1+\mu_2} \right ) , \quad n=0,1,2, \ldots,
	\end{align*}
	for some $b>1$, $K>0$ and $\mu_2\geq \mu_1>0$. If
	\begin{align*}
		Z_0 \leq \min \left(1,(2K)^{-\frac{1}{\mu_1}} b^{-\frac{1}{\mu_1^2}}\right)
	\end{align*}
	or
	\begin{align*}
		Z_0 \leq \min  \left((2K)^{-\frac{1}{\mu_1}} b^{-\frac{1}{\mu_1^2}}, (2K)^{-\frac{1}{\mu_2}}b^{-\frac{1}{\mu_1 \mu_2}-\frac{\mu_2-\mu_1}{\mu_2^2}}\right),
	\end{align*}
	then $Z_n \leq 1$ for some $n \in \N \cup \{0\}$. Moreover,
	\begin{align*}
		Z_n \leq \min \left(1,(2K)^{-\frac{1}{\mu_1}} b^{-\frac{1}{\mu_1^2}} b^{-\frac{n}{\mu_1}}\right), \quad \text{ for all }n \geq n_0,
	\end{align*}
	where $n_0$ is the smallest $n \in \N \cup \{0\}$ satisfying $Z_n \leq 1$. In particular, $Z_n \to 0$ as $n \to \infty$.
\end{lemma}

In what follows we write
\begin{align*}
	v_+ := \max \{v,0\}\quad \text{and}\quad v_- := \max \{-v,0\}.
\end{align*}
Moreover, we denote by $|E|$ the $N$-dimensional Lebesgue measure of $E\subset\R^N$.

For $p \in C_+(\close)$ we redefine the critical variable exponents to $p$ by
\begin{align*}
	p^*(x)=
	\begin{cases}
		\hat{p}^*(x) & \text{if }p^+<N,\\
		q_1(x) & \text{if } N \leq p^+,
	\end{cases} \quad\text{for all }x\in\close
\end{align*}
and
\begin{align*}
	p_*(x)=
	\begin{cases}
		\hat{p}_*(x) & \text{if }p^+<N,\\
		q_2(x)  & \text{if } N \leq p^+,
	\end{cases} \quad\text{for all }x\in\close,
\end{align*}
where $q_1, q_2 \in C(\close)$ are arbitrarily chosen such that $p(x) < q_1(x)\leq \hat{p}^*(x)$ and $p(x) < q_2(x)\leq \hat{p}_*(x)$ for all $x \in \close$.

\section{The boundedness of solutions for the Dirichlet problem}\label{section_3}

In this section, we study the boundedness of weak solutions for the Dirichlet problem given in \eqref{D}. We suppose the following structure conditions on the data.
\begin{enumerate}
	\item[H(p1):]
		$p\in C^{0, \frac{1}{|\log t|}}(\close) \cap C_+(\close)$.
	\item[H(D):]
		The functions $\mathcal{A}\colon\Omega\times\R\times\R^N\to \R^N$ and $\mathcal{B}\colon\Omega \times \R\times \R^N\to \R$ are Carath\'eodory functions such that
		\begin{align*}
			\text{(A1)}\quad & |\mathcal{A}(x,s,\xi)| \leq a_1|\xi|^{p(x)-1}+a_2|s|^{p^*(x)\frac{p(x)-1}{p(x)}}+a_3,\\
			\text{(A2)}\quad & \mathcal{A}(x,s,\xi)\cdot \xi \geq a_4|\xi|^{p(x)}-a_5|s|^{p^*(x)}-a_6,\\
			\text{(B)}\quad & |\mathcal{B}(x,s,\xi)| \leq b_1|\xi|^{p(x) \frac{p^*(x)-1}{p^*(x)}}+b_2|s|^{p^*(x)-1}+b_3,
		\end{align*}
		for a.\,a.\,$x \in \Omega$, for all $s \in \R$, for all $\xi \in \R^N$ and with positive constants $a_j, j\in \{1,\ldots,6\}$ and $b_\ell, \ell \in \{1,2,3\}$.
\end{enumerate}

A function $u\in W_0^{1,p(\cdot)}(\Omega) $ is called a weak solution of problem \eqref {D}  if
\begin{align} \label{def_sol_D}
	\int_{\Omega}\mathcal{A}(x,u,\nabla u)\cdot \nabla \varphi \diff x
	= \int_{\Omega}\mathcal{B}(x,u, \nabla u)\varphi \diff x
\end{align}
is satisfied for all test functions $\varphi\in W_0^{1,p(\cdot)}(\Omega) $. By Proposition \ref{embedding_critical} along with the hypotheses H(D) and H(p1) it is easy to see that the definition above is well-defined.

The main result in this section is the following theorem about the boundedness of weak solutions of problem \eqref{D}. The proof is based on ideas of the paper of Winkert-Zacher \cite{Winkert-Zacher-2012} by applying De Giorgi's iteration technique along with the localization method.

\begin{theorem} \label{Theo.D}
	Let hypotheses H(D) and H(p1) be satisfied. Then, any weak solution of problem \eqref{D} is of class $L^\infty (\Omega)$.
\end{theorem}

\begin{proof}
	The compactness of $\overline{\Omega}$  implies that, for any $R>0$, there exists a finite open cover $\{B_i(R)\}_{i=1,\ldots, m}$ of balls $B_i:=B_i(R)$ with radius $R$ such that $\overline{\Omega} \subset \bigcup\limits_{i=1}^{m}B_{i} $ (see, for example, Papageorgiou-Winkert \cite[1.4.86]{Papageorgiou-Winkert-2018}) and each $\Omega_i:=B_i\cap\Omega$ ($i=1,\cdots,m$) is a Lipschitz domain as well. We may take $R$ sufficiently small such that
	\begin{align}\label{D.loc.exp}
	p_i^+:= \max_{x\in \overline{B}_{i} \cap  \close}p(x)
	<{(p^*)}^{-}_{i}:=\min_{x\in \overline{B}_{i} \cap  \close} p^*(x)\quad \text{for all }i\in\{1,\cdots,m\}.
	\end{align}
	Let $u$ be a weak solution of problem \eqref{D}. Let $\kappa_*\geq 1$ be sufficiently large such that
	\begin{equation}\label{k*}
		\int_{A_{\kappa_*}}|\nabla u|^{p(x)}\diff x+\int_{A_{\kappa_*}}u^{p^*(x)}\diff x<1,
	\end{equation}
	where $A_\kappa:=\{x\in\Omega\,:\, u(x)>\kappa\}$ for $\kappa\in\R$.
	
	We define
	\begin{equation}\label{Zn.def}
		Z_n:=\int_{A_{\kappa_n}}|\nabla u|^{p(x)}\diff x+\int_{A_{\kappa_n} }(u-\kappa_{n})^{p^*(x)}\diff x,
	\end{equation}
	where
	\begin{equation*}
		\kappa_n:=\kappa_*\left(2-\frac{1}{2^n}\right)
		\quad \text{for }n \in \N_0=\{0,1,2,\ldots\}.
	\end{equation*}
	Obviously, it holds
	\begin{equation}\label{k_n}
		\kappa_n \nearrow 2\kappa_* \quad \text{and} \quad \kappa_* \leq \kappa_n <2\kappa_*  \quad \text{for all }n\in \mathbb{N}_0.
	\end{equation}
	Since $\kappa_{n}<\kappa_{n+1}$ and $A_{\kappa_{n+1}}\subset A_{\kappa_n}$ for all $n\in\N_0$, we have
	\begin{equation}\label{Zn.decreasing}
		Z_{n+1}\leq Z_n\quad\text{ for all } n\in \N_0.
	\end{equation}
	Moreover, for $x \in A_{\kappa_{n+1}}$ we see that
	\begin{equation*}
		u(x)- \kappa_n \geq u(x)\l(1-\frac{\kappa_n}{\kappa_{n+1}}\r) = \frac{u(x)}{2^{n+2}-1}.
	\end{equation*}
	Hence, we obtain
	\begin{equation} \label{est.u}
		u(x)\leq (2^{n+2}-1)(u(x)-\kappa_n)\quad \text{for a.\,a.\,}x\in A_{\kappa_{n+1}}\text{ and for all }n\in\N_0.
	\end{equation}
	Furthermore, we have
	\begin{align}\label{|A_{k_{n+1}}|}
		\begin{split}
			|A_{\kappa_{n+1}}|
			& \leq \int_{A_{\kappa_{n+1}}} \left (\frac{u-\kappa_n}{\kappa_{n+1}-\kappa_n}\right )^{p^*(x)} \diff x\\
			& \leq \int_{A_{\kappa_{n}}} \frac{2^{p^*(x)(n+1)}}{\kappa_*^{p^*(x)}} (u-\kappa_n)^{p^*(x)}\diff x \\
			& \leq \frac{2^{(p^*)^+(n+1)}}{\kappa_*^{(p^*)^-}} \int_{A_{\kappa_{n}}} (u-\kappa_{n})^{p^*(x)}\diff x \\
			& = \frac{2^{(p^*)^+(n+1)}}{\kappa_*^{(p^*)^-}} Z_n\\
			& \leq 2^{(p^*)^+(n+1)} Z_n\quad \text{for all }n\in\N_0.
		\end{split}
	\end{align}

	In the rest of the proof, we denote by $C_i$ ($i\in\N$) positive constants which are independent of $n$ and $\kappa_*$.	

	{\bf Claim 1:}
	There exist positive constants $\mu_1,\mu_2$ such that
	\begin{equation*}
		\int_{A_{\kappa_{n+1}}}(u-\kappa_{n+1})^{p^*(x)}\diff x
		\leq  C_1 2^{\frac{n\left((p^*)^+\right)^2}{p^-}}\left(Z_{n}^{1+\mu_1}+Z_{n}^{1+\mu_2}\right)\quad \text{for all }n\in\N_0.
	\end{equation*}
	
	First, note that
	\begin{align}\label{decompose1}
		\begin{split}
			\int_{A_{\kappa_{n+1}}}(u-\kappa_{n+1})^{p^*(x)}\diff x
			&=\int_{\Omega}(u-\kappa_{n+1})_+^{p^*(x)}\diff x\\
			&\leq \sum_{i=1}^m\int_{\Omega_{i}}(u-\kappa_{n+1})_+^{p^*(x)}\diff x.
		\end{split}
	\end{align}
	Let $i\in\{1,\cdots,m\}$. From \eqref{k*}, \eqref{k_n}, Proposition \ref{norm-modular}(i), (iii) and Proposition \ref{embedding_critical} for $\Omega=\Omega_{i}$ we have
	\begin{align*}
		&\int_{\Omega_{i}}(u-\kappa_{n+1})_+^{p^*(x)}\diff x\\
		&\leq \|(u-\kappa_{n+1})_+\|_{L^{p^*(\cdot)}(\Omega_{i})}^{(p^*)_i^-}\\
		& \leq C_2\left[\|\nabla (u-\kappa_{n+1})_+\|_{L^{p(\cdot)}(\Omega_{i})}+\|(u-\kappa_{n+1})_+\|_{L^{p(\cdot)}(\Omega_{i})}\right]^{(p^*)_i^-}.
	\end{align*}
	Then, again by Proposition \ref{norm-modular} along with \eqref{equi.norms} and \eqref{k*} this leads to
	\begin{align*}
		&\int_{\Omega_{i}}(u-\kappa_{n+1})_+^{p^*(x)}\diff x\\
		&\leq C_3\left(\int_{\Omega_{i}}|\nabla (u-\kappa_{n+1})_+|^{p(x)}\diff x+\int_{\Omega_{i}}(u-\kappa_{n+1})_+^{p(x)}\diff x\right)^{\frac{(p^*)_i^-}{p_i^+}}\\
		&\leq C_3\left(\int_{A_{\kappa_{n+1}}}|\nabla u|^{p(x)}\diff x+\int_{A_{\kappa_{n+1}}}(u-\kappa_{n+1})^{p^*(x)}\diff x+|A_{\kappa_{n+1}}|\right)^{\frac{(p^*)_i^-}{p_i^+}}.
	\end{align*}
	From this, \eqref{Zn.def}, \eqref{Zn.decreasing} and \eqref{|A_{k_{n+1}}|} we obtain
	\begin{align*}
		\int_{\Omega_{i}}(u-\kappa_{n+1})_+^{p^*(x)}\diff x
		\leq C_42^{\frac{n(p^*)^+(p^*)_i^-}{p_i^+}}Z_{n}^{\frac{(p^*)_i^-}{p_i^+}}.
	\end{align*}
	Combining this with \eqref{decompose1} gives
	\begin{equation*}
		\int_{\Omega}(u-\kappa_{n+1})_+^{p^*(x)}\diff x
		\leq C_52^{\frac{n\left((p^*)^+\right)^2}{p^-}}\left(Z_{n}^{1+\mu_1}+Z_{n}^{1+\mu_2}\right),
	\end{equation*}
	where
	\begin{align*}
		0<\mu_1:=\min_{1\leq i\leq m} \frac{(p^*)_i^-}{p_i^+}-1\leq\mu_2:=\max_{1\leq i\leq m} \frac{(p^*)_i^-}{p_i^+}-1
	\end{align*}
	 in view of \eqref{D.loc.exp}. This proves Claim 1.
	
	{\bf Claim 2:} It holds that
	\begin{equation*}
		\int_{A_{\kappa_{n+1}}}|\nabla u|^{p(x)}\diff x\leq  C_6 2^{n\l[\frac{\left((p^*)^+\right)^2}{p^-}+(p^*)^+\r]}\left(Z_{n-1}^{1+\mu_1}+Z_{n-1}^{1+\mu_2}\right)\quad \text{for all }n\in\N.
	\end{equation*}
	Testing \eqref{def_sol_D} with $\varphi =(u-\kappa_{n+1})_{+} \in W_0^{1,p(\cdot)}(\Omega)$ yields
	\begin{align*}
		\int_{\Omega} \mathcal{A}(x,u,\nabla u) \cdot \nabla(u-\kappa_{n+1})_{+} \diff x
		=\int_{\Omega }\mathcal{B}(x,u,\nabla u)(u-\kappa_{n+1})_{+} \diff x,
	\end{align*}
	which can be written as
	\begin{equation}\label{D.var.Eq}
		\int_{A_{\kappa_{n+1}}} \mathcal{A}(x,u,\nabla u) \cdot \nabla u \diff x
		=\int_{A_{\kappa_{n+1}} }\mathcal{B}(x,u,\nabla u)(u-\kappa_{n+1}) \diff x.
	\end{equation}
	Since $u\geq u-\kappa_{n+1} >0$ and $u> \kappa_{n+1} \geq 1$ on $A_{\kappa_{n+1}},$ using hypotheses (A2) and (B) and taking into account Young's inequality with $\eps \in (0,1]$, we estimate each term in \eqref{D.var.Eq} in order to get
	\begin{align*}
		&\int_{A_{\kappa_{n+1}}} \mathcal{A}(x,u,\nabla u)\cdot \nabla u \diff x\\
		&\geq a_4\int_{A_{\kappa_{n+1}} }|\nabla u|^{p(x)} \diff x-a_5\int_{A_{\kappa_{n+1}} }u^{p^*(x)}\diff x-a_6\int_{A_{\kappa_{n+1}} }1\diff x \\
		&\geq a_4\int_{A_{\kappa_{n+1}} }|\nabla u|^{p(x)} \diff x-\max\{a_5,a_6\}\int_{A_{\kappa_{n+1}} }u^{p^*(x)}\diff x,
	\end{align*}
	and
	\begin{align*}
		&\int_{A_{\kappa_{n+1}}} \mathcal{B}(x,u,\nabla u)(u-\kappa_{n+1})\diff x\\
		& \leq b_1\int_{A_{\kappa_{n+1}} }|\nabla u|^{p(x) \frac{p^*(x)-1}{p^*(x)}}u\diff x+b_2\int_{A_{\kappa_{n+1}}}u^{p^*(x)}\diff x +b_3 \int_{A_{\kappa_{n+1}}} 1 \diff x\\
		& \leq b_1\int_{A_{\kappa_{n+1}}} \left [ \eps^{\frac{p^*(x)-1}{p^*(x)}} |\nabla u|^{p(x)\frac{p^*(x)-1}{p^*(x)}} \eps^{-\frac{p^*(x)-1}{p^*(x)}}u \right ] \diff x+C_7\int_{A_{\kappa_{n+1}}}u^{p^*(x)}\diff x \\
		&\leq  b_1 \int_{A_{\kappa_{n+1}}} \eps |\nabla u|^{p(x)} \diff x + b_1 \int_{A_{\kappa_{n+1}}} \eps^{-(p^*(x)-1)} u^{p^*(x)}\diff x+C_7\int_{A_{\kappa_{n+1}} }u^{p^*(x)}\diff x\\
		&\leq  \eps b_1 \int_{A_{\kappa_{n+1}}}  |\nabla u|^{p(x)} \diff x + \l(b_1 \eps^{-((p^*)^+-1)}+C_7\r) \int_{A_{\kappa_{n+1}}}  u^{p^*(x)}\diff x.
	\end{align*}
	Taking $\eps=\min \{1,\frac{a_4}{2b_1}\}$ and combining the estimates above with \eqref{D.var.Eq} and then using \eqref{est.u}, we obtain
	\begin{align*}
		\int_{A_{\kappa_{n+1}}}|\nabla u|^{p(x)}\diff x &
		\leq  C_8\int_{A_{\kappa_{n+1}}}u^{p^*(x)}\diff x\\
		& \leq C_8\int_{A_{\kappa_{n+1}}}\left[(2^{n+2}-1)(u-\kappa_n)\right]^{p^*(x)}\diff x.
	\end{align*}
	Hence,
	\begin{equation*}
		\int_{A_{\kappa_{n+1}}}|\nabla u|^{p(x)}\diff x
		\leq C_92^{n(p^*)^+}\int_{\Omega}(u-\kappa_n)_+^{p^*(x)}\diff x.
	\end{equation*}
	Then, Claim 2 follows from the last inequality and Claim 1.

	From Claims 1 and 2 along with \eqref{Zn.decreasing} we conclude that
	\begin{equation}\label{Recur}
		Z_{n+1}\leq C_{10} b^n\left(Z_{n-1}^{1+\mu_1}+Z_{n-1}^{1+\mu_2}\right)\quad \text{for all }n\in\N,
	\end{equation}
	where
	\begin{align*}
		b:=2^{\big[\frac{\left((p^*)^+\right)^2}{p^-}+(p^*)^+\big]}>1.
	\end{align*}
	This yields
	\begin{equation*}
		Z_{2(n+1)}\leq C_{10} b^{2n+1}\left(Z_{2n}^{1+\mu_1}+Z_{2n}^{1+\mu_2}\right)\quad \text{for all }n\in\N_0,
	\end{equation*}
	that is,
	\begin{equation}\label{Recur1}
		\tilde{Z}_{n+1}\leq bC_{10} \tilde{b}^n\left(\tilde{Z}_n^{1+\mu_1}+\tilde{Z}_n^{1+\mu_2}\right) \quad \text{for all }n\in\N_0,
	\end{equation}
	where $\tilde{Z}_n:=Z_{2n}$ and $\tilde{b}:=b^2$. Applying Lemma \ref{leRecur} to \eqref{Recur1} yields
	\begin{equation}\label{Recur+1}
		Z_{2n}=\tilde{Z}_n \to 0 \quad \text {as }  n\to \infty
	\end{equation}
	provided that
	\begin{equation}\label{Z_0}
		\tilde{Z}_{0}\leq \min\left\{(2bC_{10})^{-\frac{1}{\mu_1}}\ \tilde{b}^{-\frac{1}{\mu_1^{2}}},\left(2bC_{10}\right)^{-\frac{1}{\mu_2}}\ \tilde{b}^{-\frac{1}{\mu_1\mu_2}-\frac{\mu_2-\mu_1}{\mu_2^{2}}}\right\}.
	\end{equation}
	Again, from \eqref{Recur} we obtain
	\begin{equation*}
		Z_{2(n+1)+1}\leq C_{10} b^{2(n+1)}\left(Z_{2n+1}^{1+\mu_1}+Z_{2n+1}^{1+\mu_2}\right)\quad \text{for all } n\in\N_0,
	\end{equation*}
	which can be written as
	\begin{equation}\label{Recur2}
		\hat{Z}_{n+1}\leq \tilde{b}C_{10} \tilde{b}^n\left(\hat{Z}_n^{1+\mu_1}+\hat{Z}_n^{1+\mu_2}\right) \quad \text{for all }n\in\N,
	\end{equation}
	where $\hat{Z}_n:=Z_{2n+1}$. From Lemma \ref{leRecur} applied to \eqref{Recur2} it follows that
	\begin{equation}\label{Recur+2}
		Z_{2n+1}=\hat{Z}_n \to 0 \quad \text{as } n\to \infty
	\end{equation}
	provided that
	\begin{equation}\label{Z_0'}
		\hat{Z}_{0}\leq \min\left\{(2\tilde{b}C_{10})^{-\frac{1}{\mu_1}}\ \tilde{b}^{-\frac{1}{\mu_1^{2}}},\left(2\tilde{b}C_{10}\right)^{-\frac{1}{\mu_2}}\ \tilde{b}^{-\frac{1}{\mu_1\mu_2}-\frac{\mu_2-\mu_1}{\mu_2^{2}}}\right\}.
	\end{equation}
	Note that
	\begin{align*}
		\hat{Z}_{0}=Z_1\leq Z_0=\tilde{Z}_0\leq \int_{A_{\kappa_*}}|\nabla u|^{p(x)}\diff x+\int_{A_{\kappa_*} }u^{p^*(x)}\diff x.
	\end{align*}
	Therefore, by choosing $\kappa_*>1$ sufficiently large we have
	\begin{align*}
		&\int_{A_{\kappa_*}}|\nabla u|^{p(x)}\diff x+\int_{A_{\kappa_*} }u^{p^*(x)}\diff x\\
		&\leq \min\left\{1,(2\tilde{b}C_{10})^{-\frac{1}{\mu_1}}\ \tilde{b}^{-\frac{1}{\mu_1^{2}}},\left(2\tilde{b}C_{10}\right)^{-\frac{1}{\mu_2}}\ \tilde{b}^{-\frac{1}{\mu_1\mu_2}-\frac{\mu_2-\mu_1}{\mu_2^{2}}}\right\}.
	\end{align*}
	Hence, \eqref{k*}, \eqref{Z_0} and \eqref{Z_0'} are fulfilled and we obtain \eqref{Recur+1} and \eqref{Recur+2}. This means that
	\begin{align*}
		Z_n=\int_{A_{\kappa_n}}|\nabla u|^{p(x)}\diff x+\int_{A_{\kappa_n} }(u-\kappa_n)^{p^*(x)}\diff x\to 0 \quad \text{as } n\to\infty.
	\end{align*}
	In particular, we have
	\begin{align*}
		\int_{\Omega}(u-2\kappa_*)_+^{p^*(x)}\ \diff x=0.
	\end{align*}
	Consequently, $(u-2\kappa_*)_{+}=0$ a.\,e.\,in $\Omega$ and so
	\begin{align*}
		\underset{\Omega}{\mathop{\rm ess\,sup}}\ u \leq 2\kappa_*.
	\end{align*}
	
	Replacing $u$ by $-u$ in the arguments above we also obtain
	\begin{align*}
		\underset{\Omega}{\mathop{\rm ess\,sup}}\ (-u) \leq 2\kappa_*.
	\end{align*}
	Hence, $\|u\|_\infty\leq 2\kappa_*$. This finishes the proof.
\end{proof}

\section{The boundedness of solutions for the Neumann problem}\label{section_4}

In this section, we study the boundedness of solutions for the Neumann problem given in \eqref{N}.
We assume the following structure conditions on the functions involved.
\begin{enumerate}
	\item[H(p2):]
		$p\in C_+(\close)\cap \Wp{\gamma}$ for some $\gamma>N$.
	\item[H(N):]
	The functions $\mathcal{A}\colon\Omega\times\R\times\R^N\to \R^N$, $\mathcal{B}\colon\Omega \times \R\times \R^N\to \R$ and $\mathcal{C}\colon \Gamma \times \R \to\R$ are Carath\'eodory functions such that
	\begin{align*}
	\text{(A1)}\quad & |\mathcal{A}(x,s,\xi)| \leq a_1|\xi|^{p(x)-1}+a_2|s|^{p^*(x)\frac{p(x)-1}{p(x)}}+a_3 && \text{for a.\,a.\,}x\in\Omega,\\
	\text{(A2)}\quad & \mathcal{A}(x,s,\xi)\cdot \xi \geq a_4|\xi|^{p(x)}-a_5|s|^{p^*(x)}-a_6 && \text{for a.\,a.\,}x\in\Omega,\\
	\text{(B)}\quad & |\mathcal{B}(x,s,\xi)| \leq b_1|\xi|^{p(x) \frac{p^*(x)-1}{p^*(x)}}+b_2|s|^{p^*(x)-1}+b_3&&\text{for a.\,a.\,}x\in\Omega,\\
	\text{(C)}\quad & |\mathcal{C}(x,s)| \leq c_1|s|^{p_*(x)-1}+c_2, &&\text{for a.\,a.\,}x\in\Gamma,
	\end{align*}
	for all $s \in \R$, for all $\xi \in \R^N$ and with positive constants $a_j, j\in \{1,\ldots,6\}$, $b_\ell, \ell \in \{1,2,3\}$ and $c_k, k\in \{1,2\}$.
\end{enumerate}

We say that $ u\in W^{1,p(\cdot)}(\Omega) $ is a weak solution of \eqref {N} if
\begin{equation}\label{def_sol_N}
	\int_{\Omega}\mathcal{A}(x,u,\nabla u)\cdot \nabla \varphi \diff x
	= \int_{\Omega}\mathcal{B}(x,u, \nabla u)\varphi \diff x+\int_{\Gamma}\mathcal{C}(x,u)\varphi \diff\sigma
\end{equation}
holds for all test functions $\varphi\in W^{1,p(\cdot)}(\Omega)$, where $\diff 
\sigma$ denotes the usual $(N-1)$-dimensional surface measure. By Propositions \ref{embedding_critical} and \ref{embedding_critical_boundary} along with the hypotheses H(N), H(p2)  and \eqref{inclusions} we verify that the definition of a weak solution of problem \eqref{N} stated in \eqref{def_sol_N} is well-defined.

The main result in this section is the following one about the boundedness of solutions of problem \eqref{N}.
\begin{theorem} \label{Theo.N}
	Let hypotheses H(N) and H(p2) be satisfied. Then, any weak solution of problem \eqref{N} is of class $L^\infty (\Omega)\cap L^\infty (\Gamma)$.
\end{theorem}

\begin{proof}
	As before, since $\overline{\Omega}$ is compact, for any $R>0$, there exists a finite open cover $\{B_i(R)\}_{i=1}^m$ of balls $B_i:=B_i(R)$ with radius $R$ such that $\overline{\Omega} \subset \bigcup\limits_{i=1}^{m}B_{i} $ and each $\Omega_i:=B_i\cap\Omega$ ($i=1,\cdots,m$) is a Lipschitz domain as well. We denote by $I$ the set of all $i\in\{1,\cdots,m\}$ such that $B_i\cap \Gamma\ne\emptyset$ and take $R>0$ sufficiently small such that
	\begin{equation*}
		p_i^+:=\max_{x\in\overline{B}_{i} \cap  \close} p(x) <{(p^*)}^{-}_{i}:=\min_{x\in \overline{B}_{i} \cap  \close} p^*(x)\quad \text{for all }i\in \{1,\cdots,m\}
	\end{equation*}
	and
	\begin{equation}\label{N.loc.exp2}
		p_i^+<{(p_*)}^{-}_{i}:=\min_{x\in \overline{B_i}\cap \Gamma} p_*(x)\quad \text{for all }i\in I.
	\end{equation}
	Let $u$ be a weak solution to problem~\eqref{N} and let $\kappa_*\geq 1$ be sufficiently large such that
	\begin{equation}\label{N.k*}
		\int_{A_{\kappa_*}}|\nabla u|^{p(x)}\diff x+\int_{A_{\kappa_*}}u^{p^*(x)}\diff x+\int_{\Gamma_{\kappa_*}}u^{p_*(x)}\diff\sigma<1,
	\end{equation}
	where
	\begin{align*}
		A_\kappa:=\{x\in\Omega \, : \, u(x)>\kappa\} \quad\text{and}\quad \Gamma_k:=\{x\in\Gamma \,:\, u(x)>\kappa\} \quad\text{for }k\in\R.
	\end{align*}
	We define
	\begin{equation}\label{N.Zn.def}
		Z_n:=\int_{A_{\kappa_n}}|\nabla u|^{p(x)}\diff x+\int_{A_{\kappa_n} }(u-\kappa_n)^{p^*(x)}\diff x+\int_{\Gamma_{\kappa_n}}(u-\kappa_n)^{p_*(x)}\diff\sigma,
	\end{equation}
	where $\{\kappa_{n}\}_{n\in\N_0}$ is defined as in the proof of Theorem \ref{Theo.D}. Since $\kappa_{n}<\kappa_{n+1}$, we have $A_{\kappa_{n+1}}\subset A_{\kappa_{n}}$ and $\Gamma_{\kappa_{n+1}}\subset \Gamma_{\kappa_{n}}$ for all $n\in\N_0$. Hence, it holds
	\begin{equation}\label{N.Zn.decreasing}
	Z_{n+1}\leq Z_n\quad\text{for all }n\in \N_0.
	\end{equation}
	Moreover, we have the following estimates (see the proof of Theorem \ref{Theo.D})
	\begin{align}
		u(x)&\leq \l(2^{n+2}-1\r)(u(x)-\kappa_{n})\quad \text{for a.\,a.\,}x\in A_{\kappa_{n+1}} \text{ and for all } n\in\N_0,\label{N.est.u1}\\
		u(x)&\leq \l(2^{n+2}-1\r)(u(x)-\kappa_{n})\quad \text{for a.\,a.\,}x\in \Gamma_{\kappa_{n+1}}\text{ and for all } n\in\N_0. \label{N.est.u2}
	\end{align}
	As shown in \eqref{|A_{k_{n+1}}|} the following estimate holds true
	\begin{equation} \label{N.|A_k|}
	|A_{\kappa_{n+1}}| \leq \frac{(p^*)^+2^{(n+1)}}{\kappa_*^{(p^*)^{-}}} Z_n\leq 2^{(n+1)(p^*)^+} Z_n\quad \text{for all }n\in\N_0.
	\end{equation}
	In the following we will denote by $C_i$ ($i\in\N$) positive constants which are independent of $n$ and $\kappa_*$.
	
	{\bf Claim 1:} It holds that
	\begin{equation*}
		\int_{A_{\kappa_{n+1}}}(u-\kappa_{n+1})^{p^*(x)}\diff x\leq  C_1 2^{\frac{n\left((p^*)^+\right)^2}{p^-}}\left(Z_{n}^{1+\nu_1}+Z_{n}^{1+\nu_2}\right)\quad \text{for all } n\in\N_0,
	\end{equation*}
	where
	\begin{align*}
		0<\nu_1:=\min_{1\leq i\leq m} \frac{(p^*)_i^-}{p_i^+}-1\leq\nu_2:=\max_{1\leq i\leq m} \frac{(p^*)_i^-}{p_i^+}-1.
	\end{align*}

	The proof of Claim 1 is the same as Claim 1 in Theorem \ref{Theo.D}.
	
	{\bf Claim 2:} There exist positive constants $\nu_3,\nu_4$ such that
	\begin{equation*}
		\int_{\Gamma_{\kappa_{n+1}}}(u-\kappa_{n+1})^{p_*(x)}\diff \sigma\leq  C_2 2^{\frac{n\left((p^*)^+\right)^2}{p^-}}\left(Z_{n}^{1+\nu_3}+Z_{n}^{1+\nu_4}\right)\quad \text{for all } n\in\N_0.
	\end{equation*}
	
	First, we see that
	\begin{align}\label{N.deco.Bdr}
		\begin{split}
			\int_{\Gamma_{\kappa_{n+1}}}(u-\kappa_{n+1})^{p_*(x)}\diff \sigma
			&=\int_{\Gamma}(u-\kappa_{n+1})_+^{p_*(x)}\diff \sigma\\
			& \leq \sum_{i\in I}\int_{\Gamma\cap \partial \Omega_{i}}(u-\kappa_{n+1})_+^{p_*(x)}\diff \sigma.
		\end{split}
	\end{align}
	Let $i\in I$. From Proposition \ref{norm-modular}, \eqref{N.k*} and  Proposition \ref{embedding_critical_boundary} for $\Omega=\Omega_{i}$ we have
	\begin{align*}
		& \int_{\Gamma\cap \partial \Omega_{i}}(u-\kappa_{n+1})_+^{p_*(x)}\diff \sigma\\
		& \leq \|(u-\kappa_{n+1})_+\|_{L^{p_*(\cdot)}(\Gamma\cap \partial \Omega_{i})}^{(p_*)_i^-}\\
		& \leq \|(u-\kappa_{n+1})_+\|_{L^{p_*(\cdot)}(\partial \Omega_{i})}^{(p_*)_i^-}\\
		&\leq C_3\left[\|\nabla (u-\kappa_{n+1})_+\|_{L^{p(\cdot)}(\Omega_{i})}+\|(u-\kappa_{n+1})_+\|_{L^{p(\cdot)}(\Omega_{i})}\right]^{(p_*)_i^-}.
	\end{align*}
	Then, by Proposition \ref{norm-modular}, \eqref{equi.norms} and \eqref{N.k*} it follows
	\begin{align*}
		& \int_{\Gamma\cap \partial \Omega_{i}}(u-\kappa_{n+1})_+^{p_*(x)}\diff \sigma\\
		&\leq C_4\left(\int_{\Omega_{i}}|\nabla (u-\kappa_{n+1})_+|^{p(x)}\diff x+\int_{\Omega_{i}}(u-\kappa_{n+1})_+^{p(x)}\diff x\right)^{\frac{(p_*)_i^-}{p_i^+}}\\
	\notag&\leq C_4\left(\int_{A_{\kappa_{n+1}}}|\nabla u|^{p(x)}\diff x+\int_{A_{\kappa_{n+1}}}(u-\kappa_{n+1})^{p^*(x)}\diff x+|A_{\kappa_{n+1}}|\right)^{\frac{(p_*)_i^-}{p_i^+}}.
	\end{align*}
	From this combined with \eqref{N.Zn.def}, \eqref{N.Zn.decreasing} as well as \eqref{N.|A_k|} we obtain
	\begin{align}\label{N.grad3}
		\int_{\Gamma\cap \partial \Omega_{i}}(u-\kappa_{n+1})_+^{p_*(x)}\diff \sigma&\leq C_52^{\frac{n(p^*)^+(p_*)_i^-}{p_i^+}}Z_{n}^{\frac{(p_*)_i^-}{p_i^+}}.
	\end{align}
	From \eqref{N.deco.Bdr} and \eqref{N.grad3} we conclude that
	\begin{equation*}
		\int_{\Gamma}(u-\kappa_{n+1})_+^{p_*(x)}\diff \sigma\leq C_62^{\frac{n\left((p^*)^+\right)^2}{p^-}}\left(Z_{n}^{1+\nu_3}+Z_{n}^{1+\nu_4}\right),
	\end{equation*}
	where
	\begin{align*}
		0<\nu_3:=\min_{i\in I} \frac{(p_*)_i^-}{p_i^+}-1\leq\nu_4:=\max_{i\in I} \frac{(p_*)_i^-}{p_i^+}-1;
	\end{align*}
	see \eqref{N.loc.exp2}. This proves Claim 2.
	
	{\bf Claim 3:} It holds that
	\begin{equation*}
		\int_{A_{\kappa_{n+1}}}|\nabla u|^{p(x)}\diff x\leq  C_7 2^{n\l[\frac{\left((p^*)^+\right)^2}{p^-}+(p^*)^+\r]}\left(Z_{n-1}^{1+\mu_1}+Z_{n-1}^{1+\mu_2}\right)\quad \text{for all }n\in\N,
	\end{equation*}
	where
	\begin{align*}
		0<\mu_1:=\min_{1\leq i\leq 4}\nu_i\leq \mu_2:=\max_{1\leq i\leq 4}\nu_i.
	\end{align*}

	Taking $\varphi =(u-\kappa_{n+1})_{+} \in W^{1,p(\cdot)}(\Omega)$ as test function in \eqref{def_sol_N} gives
	\begin{align*}
		\int_{\Omega}
		\mathcal{A}(x,u,\nabla u) \cdot \nabla
		\varphi \diff x =\int_{\Omega }\mathcal{B}(x,u,\nabla u)\varphi \diff x+\int_{\Gamma}\mathcal{C}(x,u)\varphi \diff\sigma,
	\end{align*}
	which can be written as
	\begin{align}\label{estimate_3}
		\begin{split}
			&\int_{A_{\kappa_{n+1}}} \mathcal{A}(x,u,\nabla u) \cdot \nabla u \diff x\\ 
			&=\int_{A_{\kappa_{n+1}} }B(x,u,\nabla u)(u-\kappa_{n+1}) \diff x +\int_{\Gamma_{\kappa_{n+1}} }C(x,u)(u-\kappa_{n+1}) \diff \sigma.
		\end{split}
	\end{align}
	As done in the proof of Theorem \ref{Theo.D}, by using (A2) and (B), we have the estimates
	\begin{align}\label{estimate_4}
		\begin{split}
			&\int_{A_{\kappa_{n+1}}} \mathcal{A}(x,u,\nabla u)\cdot \nabla u \diff x\\
			&\geq a_4\int_{A_{\kappa_{n+1}} }|\nabla u|^{p(x)} \diff x-\max\{a_5,a_6\}\int_{A_{\kappa_{n+1}} }u^{p^*(x)}\diff x,
		\end{split}
	\end{align}
	and
	\begin{align}\label{estimate_5}
		\begin{split}
			&\int_{A_{\kappa_{n+1}}} \mathcal{B}(x,u,\nabla u)(u-\kappa_{n+1})\diff x\\
			&\leq  \eps b_1 \int_{A_{\kappa_{n+1}}}  |\nabla u|^{p(x)} \diff x + \l(b_1 \eps^{-((p^*)^+-1)}+C_7\r) \int_{A_{\kappa_{n+1}}}  u^{p^*(x)}\diff x
		\end{split}
	\end{align}
	with $\eps \in (0,1]$. Note that $u\geq u-\kappa_{n+1} >0$ and $u> \kappa_{n+1} \geq 1$ on $\Gamma_{\kappa_{n+1}}$. So, applying assumption (C) yields
	\begin{align}\label{estimate_6}
		\begin{split}
			\int_{\Gamma_{\kappa_{n+1}} }C(x,u)(u-\kappa_{n+1}) \diff\sigma
			& \leq c_1 \int_{\Gamma_{\kappa_{n+1}} }u^{p_*(x)} \diff\sigma+c_2 \int_{\Gamma_{\kappa_{n+1}} }1 \diff\sigma\\
			& \leq (c_1+c_2)\int_{\Gamma_{\kappa_{n+1}} }u^{p_*(x)} \diff\sigma.
		\end{split}
	\end{align}
	Combining \eqref{estimate_3}, \eqref{estimate_4}, \eqref{estimate_5} and \eqref{estimate_6}, taking $\eps=\min \{1,\frac{a_4}{2b_1}\}$ and using \eqref{N.est.u1} as well as \eqref{N.est.u2}, we obtain
	\begin{align*}
		&\int_{A_{\kappa_{n+1}}}|\nabla u|^{p(x)}\diff x\\
		& \leq  C_8\int_{A_{\kappa_{n+1}}}u^{p^*(x)}\diff x+C_8\int_{\Gamma_{\kappa_{n+1}}}u^{p_*(x)}\diff \sigma\\
		&\leq C_8\int_{A_{\kappa_{n+1}}}\left[\l(2^{n+2}-1\r)(u-\kappa_{n})\right]^{p^*(x)}\diff x\\
		&\qquad +C_8\int_{\Gamma_{\kappa_{n+1}}}\left[\l(2^{n+2}-1\r)(u-\kappa_{n})\right]^{p_*(x)}\diff \sigma.
	\end{align*}
	Hence,
	\begin{equation*}
		\int_{A_{\kappa_{n+1}}}|\nabla u|^{p(x)}\diff x
		\leq C_92^{n(p^*)^+}\left[\int_{A_{\kappa_{n}}}(u-\kappa_{n})^{p^*(x)}\diff x+\int_{\Gamma_{\kappa_n}}(u-\kappa_{n})^{p_*(x)}\diff \sigma\right].
	\end{equation*}
	Then, Claim 3 follows from the last inequality and Claims 1 and 2.
	
	From Claims 1, 2 and 3 along with \eqref{N.Zn.decreasing} one has
	\begin{equation}\label{N.Recur}
		Z_{n+1}\leq C_{10} b^n\left(Z_{n-1}^{1+\mu_1}+Z_{n-1}^{1+\mu_2}\right)\quad \text{for all } n\in\N,
	\end{equation}
	where
	\begin{align*}
		b:=2^{\l[\frac{\left((p^*)^+\right)^2}{p^-}+(p^*)^+\r]}>1.
	\end{align*}
	Repeating now the same arguments used in the proof of Theorem \ref{Theo.D}, by choosing $\kappa_*>1$ sufficiently large such that
	\begin{align*}
		&\int_{A_{\kappa_*}}|\nabla u|^{p(x)}\diff x+\int_{A_{\kappa_*} }u^{p^*(x)}\diff x+\int_{\Gamma_{\kappa_*} }u^{p_*(x)}\diff \sigma\\
		&\leq \min\left\{1,(2\tilde{b}C_{10})^{-\frac{1}{\mu_1}}\ \tilde{b}^{-\frac{1}{\mu_1^{2}}},\left(2\tilde{b}C_{10}\right)^{-\frac{1}{\mu_2}}\ \tilde{b}^{-\frac{1}{\mu_1\mu_2}-\frac{\mu_2-\mu_1}{\mu_2^{2}}}\right\},
	\end{align*}
	where $\tilde{b}:=b^2$, we deduce from \eqref{N.Recur} that
	\begin{align*}
		Z_n=\int_{A_{\kappa_{n}}}|\nabla u|^{p(x)}\diff x+\int_{A_{\kappa_{n}} }(u-\kappa_{n})^{p^*(x)}\diff x+\int_{\Gamma_{\kappa_{n}} }(u-\kappa_{n})^{p_*(x)}\diff \sigma\to 0
	\end{align*}
	as $n\to \infty$. This implies that
	\begin{align*}
		\int_{\Omega}(u-2\kappa_*)_+^{p^*(x)} \diff x+\int_{\Gamma}(u-2\kappa_*)_+^{p_*(x)} \diff \sigma=0.
	\end{align*}
	Therefore, $(u-2\kappa_*)_{+}=0$ a.\,e.\,in $\Omega$ and $(u-2\kappa_*)_{+}=0$ a.\,e.\,on $\Gamma$. This means that
	\begin{align*}
		\underset{\Omega}{\mathop{\rm ess\,sup}}\ u +\underset{\Gamma}{\mathop{\rm ess\,sup}}\ u\leq 4\kappa_*.
	\end{align*}

	Replacing $u$ by $-u$ in the arguments above we can show in the same way that
	\begin{align*}
		\underset{\Omega}{\mathop{\rm ess\,sup}}\ (-u) +\underset{\Gamma}{\mathop{\rm ess\,sup}}\ (-u)\leq 4\kappa_*.
	\end{align*}
	Hence,
	\begin{align*}
		\|u\|_{L^\infty(\Omega)}+\|u\|_{L^\infty(\Gamma)}\leq 4\kappa_*.
	\end{align*}
	The proof is finished.
\end{proof}

\section{The H\"older continuity}\label{section_5}

In this section we are concerning with the H\"older continuity of the solutions of the  problems \eqref{D} and \eqref{N}, respectively. In order to prove this we use the results from Sections \ref{section_3} and \ref{section_4}, respectively, along with the work of Fan-Zhao \cite{Fan-Zhao-1999}. A direct consequence is the H\"older continuity of the gradients of the solutions based on the paper of Fan \cite{Fan-2007}. To be more precise, we are going the prove the following results.

\begin{theorem}\label{H.D}
	Let hypotheses H(D) and H(p1) be satisfied. Then, any weak solution of problem \eqref{D} is of class $C^{0,\alpha}(\close)$ for some $\alpha\in (0,1]$.
\end{theorem}

\begin{theorem}\label{H.N}
	Let hypotheses H(N) and H(p2) be satisfied. Then, for any weak solution $u$ of problem~\eqref{N} it holds that $u\in C^{0,\alpha}(\Omega)$ for some $\alpha\in (0,1)$. If in addition $u\in C^{0,\beta_1}(\Gamma)$ for some $\beta_1 \in (0,1)$, then $u\in C^{0,\beta_2}(\close)$ for some $\beta_2\in (0,1)$.
\end{theorem}

We will only prove Theorem \ref{H.N} since the proof of Theorem \ref{H.D} is similar and simpler.
Let hypotheses H(N) and H(p2) be satisfied and let $u\in W^{1,p(\cdot)}(\Omega)$ be a weak solution of problem \eqref{N}. From Theorem \ref{Theo.N} we know that $u\in L^\infty\left(\Omega\right)\cap L^\infty\left(\Gamma\right)$. We fix $x_0\in\overline{\Omega}$ and denote by $B_r$ the ball centered at $x_0$ with radius $r>0$. Moreover, we set
\begin{align*}
	\Omega_r:=B_r\cap \Omega \quad \text{and}\quad A(\kappa,r):=\{x\in \Omega_r\,:\, u(x)>\kappa\} \quad\text{for } \kappa \in\R
\end{align*}
and
\begin{align*}
	M_0:=\left\|u\right\|_{L^\infty\left(\Omega\right)}+\left\|u\right\|_{L^\infty\left(\Gamma\right)}.
\end{align*}
Further, we denote by $C$ a positive constant that depends only on $M_0$ and the data and $C$ can be different in different places.

The following Caccioppoli-type inequality is essential in the proof of the H\"older continuity based on De Giorgi's iteration method.

\begin{lemma}\label{H.Cac.D}
	Let $x_0\in\Omega$. Then it holds
	\begin{align*}
		\int_{A\left(\kappa,r\right)}\left|\nabla u\right|^{p(x)}\diff x \leq C\int_{A\left(\kappa,R\right)}\left(\frac{u-\kappa}{R-r}\right)^{p(x)}\diff x + C\left|A\left(\kappa,R\right)\right|,
	\end{align*}
	for arbitrary $0<r<R$ with $B_R\subset\Omega$ and $\kappa\geq \underset{B_R}{\esssup}\, u-2M_0$. The conclusion remains valid if we replace $u$ by $-u$.
\end{lemma}

\begin{proof}
	Let $0<r<R$ with $B_R\subset\Omega$ and and let $\kappa\in\R$ be such that
	\begin{align*}
		\kappa\geq \underset{B_R}{\esssup}\, u-2M_0.
	\end{align*}
	Let $\eta\in C_c^\infty\left(\mathbb{R}^N\right)$ be such that	
	\begin{align}\label{P.Le4.3.eta}
		0\leq \eta\leq 1,\quad \operatorname{supp}(\eta)\subset B_R,\quad\eta\equiv 1\,\,\text{on}\,\,B_r\quad\text{and}\quad\left|\nabla\eta\right|\leq \frac{4}{R-r}.
	\end{align}
	Setting $v=\left(u-\kappa\right)_+$ and testing \eqref{def_sol_N} with $\varphi=v\eta^{p^+}$ gives
	\begin{align}\label{P.Le2.test}
		\begin{split}
			&\int_\Omega \left(\mathcal{A}\left(x,u,\nabla u\right)\cdot\nabla v\right)\eta^{p^+}\diff x\\
			&= -p^+\int_\Omega \left(\mathcal{A}\left(x,u,\nabla u\right)\cdot\nabla \eta \right)v\eta^{p^+-1}\diff x + \int_\Omega \mathcal{B}\left(x,u,\nabla u\right)v\eta^{p^+}\diff x.
		\end{split}
	\end{align}
	Here we have used the fact that $v\eta^{p^+}\in W_0^{1,p(\cdot)}(\Omega)$ and hence $\int_{\Gamma}\mathcal{C}(x,u)v\eta^{p^+} \diff\sigma  =0$. We now estimate each term on both sides of \eqref{P.Le2.test}. From assumption (A2) we have

	\begin{align}\label{test.Est1}
		\begin{split}
			&\int_\Omega \left(A\left(x,u,\nabla u\right)\cdot\nabla v\right)\eta^{p^+}\diff x\\
			&= \int_{A\left(\kappa,R\right)} \left(A\left(x,u,\nabla u\right)\cdot\nabla u\right)\eta^{p^+}\diff x\\
			&\geq  a_4\int_{A\left(\kappa,R\right)}\left|\nabla u\right|^{p(x)}\eta^{p^+}\diff x -a_5 \int_{A\left(\kappa,R\right)}|u|^{p^*(x)}\diff x -a_6 \int_{A\left(\kappa,R\right)} 1\diff x\\
			&\geq a_4\int_{A\left(\kappa,R\right)}\left|\nabla u\right|^{p(x)}\eta^{p^+}\diff x -(a_5+a_6)\left(M_0^{(p^*)^+}+2\right)|A\left(\kappa,R\right)|\\
			& \geq a_4\int_{A\left(\kappa,R\right)}\left|\nabla u\right|^{p(x)}\eta^{p^+}\diff x-C\left|A\left(\kappa,R\right)\right|.
		\end{split}
	\end{align}
	Now we can estimate the first term on the right hand side of \eqref{P.Le2.test} via hypothesis (A1). We obtain
	\begin{align}\label{test.Est2'}
		\begin{split}
			&-p^+\int_{A\left(\kappa,R\right)}\left(A\left(x,u,\nabla u\right)\cdot\nabla\eta\right)v\eta^{p^+-1}\diff x\\
			& \leq  p^+a_1\int_{A\left(\kappa,R\right)}\left|\nabla u\right|^{p(x)-1}\left|\nabla\eta\right|v\eta^{p^+-1}\diff x \\
			& \quad + p^+a_2\int_{A\left(\kappa,R\right)}\left|u\right|^{p^*(x)\frac{p(x)-1}{p(x)}}\left|\nabla \eta\right|v\eta^{p^+-1}\diff x\\
			& \quad + p^+a_3\int_{A\left(\kappa,R\right)}\left|\nabla \eta\right|v\eta^{p^+-1}\diff x.
		\end{split}
	\end{align}
	The first term on the right-hand side of \eqref{test.Est2'} can be estimated via Young's inequality in order to get
	\begin{align}\label{test.Est2''}
		\begin{split}
			&p^+a_1\int_{A\left(\kappa,R\right)}\left|\nabla u\right|^{p(x)-1}\left|\nabla\eta\right|v\eta^{p^+-1}\diff x\\
			&=p^+a_1\int_{A\left(\kappa,R\right)}\l(\frac{a_4 }{3a_1p^+}\r)^{\frac{p(x)-1}{p(x)}}\left|\nabla u\right|^{p(x)-1}\eta^{p^+-1} \\
			& \quad \times \l(\frac{a_4 }{3a_1p^+}\r)^{-\frac{p(x)-1}{p(x)}}\left|  \nabla\eta\right|v\diff x\\
			&\leq \frac{a_4}{3}\int_{A\left(\kappa,R\right)}\left|\nabla u\right|^{p(x)}\eta^{(p^+-1)\frac{p(x)}{p(x)-1}}\diff x+C\int_{A\left(\kappa,R\right)}\left|\nabla \eta\right|^{p(x)}v^{p(x)}\diff x\\
			&\leq \frac{a_4}{3}\int_{A\left(\kappa,R\right)}\left|\nabla u\right|^{p(x)}\eta^{p^+}\diff x+C\int_{A\left(\kappa,R\right)}\left|\nabla \eta\right|^{p(x)}v^{p(x)}\diff x.
		\end{split}
	\end{align}
	For the second and the third term on the right-hand side of \eqref{test.Est2'} we have
	\begin{align}\label{test.Est2'''}
		\begin{split}
			& p^+a_2\int_{A\left(\kappa,R\right)}\left|u\right|^{p^*(x)\frac{p(x)-1}{p(x)}}\left|\nabla \eta\right|v\eta^{p^+-1}\diff x\\
			& \leq p^+a_2\left(M_0^{\l(p^*(x)\frac{p(x)-1}{p(x)}\r)^+}+1\right) \int_{A\left(\kappa,R\right)}\left|\nabla\eta\right|v\diff x\\
			&\leq C\left|\nabla \eta\right|^{p(x)}v^{p(x)}\diff x+C|A\left(\kappa,R\right)|
		\end{split}
	\end{align}
	and
	\begin{align}\label{test.Est2''''}
		p^+a_3\int_{A\left(\kappa,R\right)}\left|\nabla \eta\right|v\eta^{p^+-1}\diff x \leq  C\int_{A\left(\kappa,R\right)}\left|\nabla \eta\right|^{p(x)}v^{p(x)}\diff x +C\left|A\left(\kappa,R\right)\right|.
	\end{align}
	Combining \eqref{test.Est2''}, \eqref{test.Est2'''} and \eqref{test.Est2''''} with \eqref{test.Est2'} leads to
	\begin{align}\label{test.Est2}
	\begin{split}
	&-p^+\int_{A\left(\kappa,R\right)}\left(A\left(x,u,\nabla u\right)\cdot\nabla\eta\right)v\eta^{p^+-1}\diff x\\
	& \leq \frac{a_4}{3}\int_{A\left(\kappa,R\right)}\left|\nabla u\right|^{p(x)}\eta^{p^+}\diff x +C\int_{A\left(\kappa,R\right)}\left|\nabla \eta\right|^{p(x)}v^{p(x)}\diff x + C\left|A\left(\kappa,R\right)\right|.
	\end{split}
	\end{align}

	We finally estimate the last term on the right hand side of \eqref{P.Le2.test}. It follows from hypothesis (B) that
	\begin{align}\label{test.Est3'}
		\begin{split}
			&\int_{\Omega}B\left(x,u,\nabla u\right)v\eta^{p^+}\diff x\\
			&=\int_{B_R}B\left(x,u,\nabla u\right)v\eta^{p^+}\diff x\\
			& \leq  b_1\int_{A\left(\kappa,R\right)}\left|\nabla u\right|^{p(x)\frac{p^*(x)-1}{p^*(x)}}v\eta^{p^+}\diff x + b_2\int_{A\left(\kappa,R\right)}\left|u\right|^{p^*(x)-1}v\eta^{p^+}\diff x\\
			&\quad + b_3|A\left(\kappa,R\right)|.
		\end{split}
	\end{align}
	Since $0<v\leq 2M_0$ on $A(\kappa,R)$ we get by applying Young's inequality
	\begin{align}\label{test.Est3''}
		\begin{split}
			&b_1\int_{A\left(\kappa,R\right)}\left|\nabla u\right|^{p(x)\frac{p^*(x)-1}{p^*(x)}}v\eta^{p^+}\diff x\\
			& \leq  2b_1M_0\int_{A\left(\kappa,R\right)}\left|\nabla u\right|^{p(x)\frac{p^*(x)-1}{p^*(x)}}\eta^{p^+}\diff x \\
			& =  2b_1M_0\int_{A\left(\kappa,R\right)} \l(\frac{a_4}{6b_1M_0}\r)^{\frac{p^*(x)-1}{p^*(x)}}\left|\nabla u\right|^{p(x)\frac{p^*(x)-1}{p^*(x)}}\eta^{p^+}\\
			&\quad\times\l(\frac{a_4}{6b_1M_0}\r)^{-\frac{p^*(x)-1}{p^*(x)}}\diff x \\
			&\leq \frac{a_4}{3}\int_{A\left(\kappa,R\right)}\left|\nabla u\right|^{p(x)}\eta^{p^+}\diff x + C\left|A\left(\kappa,R\right)\right|
		\end{split}
	\end{align}
	and
	\begin{equation}\label{test.Est3'''}
		b_2\int_{A\left(\kappa,R\right)}\left|u\right|^{p^*(x)-1}v\eta^{p^+}\diff x\leq 2b_2 (M_0+1)^{(p^*)^+}\left|A\left(\kappa,R\right)\right|.
	\end{equation}
	From \eqref{test.Est3'}, \eqref{test.Est3''} and \eqref{test.Est3'''} we obtain
	\begin{equation}\label{test.Est3}
	\int_{\Omega}B\left(x,u,\nabla u\right)v\eta^{p^+}\diff x
	\leq  \frac{a_4}{3}\int_{A\left(\kappa,R\right)}\left|\nabla u\right|^{p(x)}\eta^{p^+}\diff x + C\left|A\left(\kappa,R\right)\right|.
	\end{equation}
	Combining \eqref{P.Le2.test}, \eqref{test.Est1}, \eqref{test.Est2} and \eqref{test.Est3}, we arrive at
	\begin{align*}
		\int_{A\left(\kappa,R\right)}\left|\nabla u\right|^{p(x)} \eta^{p^+}\diff x
		\leq C\int_{A\left(\kappa,R\right)}\left|\nabla \eta\right|^{p(x)}v^{p(x)}\diff x + C\left|A\left(\kappa,R\right)\right|.
	\end{align*}
	From this and \eqref{P.Le4.3.eta} we conclude that	
	\begin{align*}
		\int_{A\left(\kappa,r\right)}\left|\nabla u\right|^{p(x)}\diff x \leq C\int_{A\left(\kappa,R\right)}\left(\frac{u-\kappa}{R-r}\right)^{p(x)}\diff x + C\left|A\left(\kappa,R\right)\right|.
	\end{align*}
	The proof is complete.
\end{proof}

Similarly we have the following result.

\begin{lemma}\label{H.Cac.N}
	Let $x_0\in \close$. Then it holds
	\begin{align}\label{Le.4.4.Ineq}
		\int_{A\left(\kappa,r\right)}\left|\nabla u\right|^{p(x)}\diff x \leq C\int_{A\left(\kappa,R\right)}\left(\frac{u-\kappa}{R-r}\right)^{p(x)}\diff x + C\left|A\left(\kappa,R\right)\right|,
	\end{align}
	for arbitrary $0<r<R$ with $B_R\cap \Gamma\ne\emptyset$ and $\kappa\geq \max\big\{\underset{\Omega_R}{\esssup}\,u-2M_0, \underset{\Gamma_R}{\esssup}\, u\big\}$, where $\Gamma_R:=\Gamma\cap B_R$. The conclusion remains valid if we replace $u$ by $-u$.
\end{lemma}

\begin{proof}
	Let $0<r<R$ with $B_R\cap \Gamma\ne\emptyset$ and let $\kappa\in\R$ be such that
	\begin{align*}
		\kappa \geq \max\big\{\underset{\Omega_R}{\esssup}\,u-2M_0, \underset{\Gamma_R}{\esssup}\, u\big\}.
	\end{align*}
	 Let $\eta$ and $v$ be as in the proof of Lemma~\ref{H.Cac.D}. Then, testing \eqref{N} with $\varphi=v\eta^{p^+}$, we get
	\begin{align}\label{P.Le4.4.test}
		\begin{split}
			&\int_\Omega \left(\mathcal{A}\left(x,u,\nabla u\right)\cdot\nabla v\right)\eta^{p^+}\diff x\\
			&= -p^+\int_\Omega \left(\mathcal{A}\left(x,u,\nabla u\right)\cdot\nabla \eta \right)v\eta^{p^+-1}\diff x \\
			& \quad+ \int_\Omega \mathcal{B}\left(x,u,\nabla u\right)v\eta^{p^+}\diff x+\int_{\Gamma}\mathcal{C}(x,u)v\eta^{p^+-1} \diff\sigma.
		\end{split}
	\end{align}
	From hypothesis (C) and Sobolev's imbedding we have
	\begin{align}\label{estimate-for-C}
		\begin{split}
			&\int_{\Gamma}\mathcal{C}(x,u)v\eta^{p^+-1} \diff\sigma\\
			&\leq c_1 \int_{\Gamma_R}|u|^{p_*(x)-1}v\eta^{p^+} \diff\sigma+c_2 \int_{\Gamma_R}v\eta^{p^+} \diff\sigma\\
			&\leq C \int_{\Gamma_R}v\eta^{p^+} \diff\sigma\\
			&\leq C \int_{\partial\Omega_R}v\eta^{p^+} \diff\sigma\\
			&\leq C\left[\int_{\Omega_R}v\eta^{p^+} \diff x+ \int_{\Omega_R}\left|\nabla\left(v\eta^{p^+}\right)\right| \diff x\right].
		\end{split}
	\end{align}
	From \eqref{estimate-for-C}, the fact that $0\leq v\leq 2M_0$ on $\Omega_R$ and applying Young's inequality we deduce
	\begin{align}\label{estimate-for-C1}
		\begin{split}
			&\int_{\Gamma}\mathcal{C}(x,u)v\eta^{p^+-1} \diff\sigma\\
			&\leq \frac{a_4}{6}\int_{A\left(\kappa,R\right)}\left|\nabla u\right|^{p(x)}\eta^{p^+}\diff x+C\int_{A\left(\kappa,R\right)}\left|\nabla \eta\right|^{p(x)}v^{p(x)}\diff x\\
			&\quad +C\left|A\left(\kappa,R\right)\right|,
		\end{split}
	\end{align}
	see \eqref{P.Le4.3.eta} and \eqref{test.Est2''''}. Combining \eqref{test.Est1}, \eqref{test.Est2}, \eqref{test.Est3}, \eqref{P.Le4.4.test} and \eqref{estimate-for-C1} gives the desired estimate in \eqref{Le.4.4.Ineq}. The proof is complete.
\end{proof}

We are now in a position to state the proof of Theorem \ref{H.N}.

\begin{proof}[Proof of Theorem \ref{H.N}]
	By Theorem \ref{Theo.N} we have that $L^\infty (\Omega)\cap L^\infty (\Gamma)$. Then by Lemma \ref{H.Cac.D} and Theorem 2.1 of Fan-Zhao \cite{Fan-Zhao-1999} we obtain $u\in C^{0,\alpha}(\Omega)$ for some $\alpha\in (0,1)$. If in addition $u\in C^{0,\beta_1}(\Gamma)$ for some $\beta_1 \in (0,1)$, then by Lemma \ref{H.Cac.N} and Theorem 2.2 of Fan-Zhao \cite{Fan-Zhao-1999} we infer that  $u\in C^{0,\beta_2}(\overline{\Omega})$ for some $\beta_2 \in (0,1)$.
\end{proof}

Finally, let us discuss the $C^{1,\alpha}$-regularity of solutions to problems \eqref{D} and \eqref{N} when the function $\mathcal{A}\colon \Omega \times \R\times \R^N\to\R^N$ satisfies further assumptions.

To this end, we suppose the following.
\begin{enumerate}
	\item [H(p3):]
		$p\in C_+(\close)\cap C^{0,\mu}(\close)$ for some $\mu \in (0,1)$;
	\item [H(A):]
		$A=(A_1,\cdots,A_N)\in C(\close\times\R\times \R^N,\R^N)$. For each $(x,s)\in \close\times\R,$ $A(x,s,\cdot)\in C^1(\R^N\setminus\{0\},\R^N)$ and there exist a nonnegative constant $k\geq 0$, a nonincreasing continuous function $\lambda \colon [0,\infty)\to (0,\infty)$ and a nondecreasing continuous function $\Lambda\colon [0,\infty)\to (0,\infty)$ such that for all $x,x_1,x_2\in \overline{\Omega}$, $s,s_1,s_2\in\R$, $\xi\in \R^N\setminus\{0\}$ and $\zeta=(\zeta_1,\cdots,\zeta_N)\in \R^N$,
		the subsequent following conditions are satisfied:
		\begin{align*}
			A(x,s,0)&=0,\\
			\sum_{i,j}\frac{\partial A_j}{\partial \xi_i}(x,s,\xi)\zeta_i\zeta_j&\geq \lambda(|s|)(k +|\xi|^2)^{\frac{p(x)-2}{2}}|\zeta|^2,\\
			\sum_{i,j}\left|\frac{\partial A_j}{\partial\xi_i}(x,s,\xi)\right|&\leq \Lambda (|s|)(k +|\xi|^2)^{\frac{p(x)-2}{2}},
		\end{align*}
		and
		\begin{align*}
			& |A(x,s_1,\xi)-A(x,s_2,\xi)|\\
			&\leq \Lambda\left(\max\{|s_1|,|s_2|\}\right)\left(|x_1-x_2|^{\mu_1}+|s_1-s_2|^{\mu_2}\right)\\
			&\quad\times \left[\left(k+|\xi|^2\right)^{\frac{p(x_1)-2}{2}}+\left(k+|\xi|^2\right)^{\frac{p(x_2)-2}{2}}\right]|\xi|\left(1+\left|\log\left(k+|\xi|^2\right)\right|\right).	
	\end{align*}
\end{enumerate}

Then, in view of Theorems \ref{Theo.D}, \ref{Theo.N} and Theorems \ref{H.D}, \ref{H.N} above along with Theorems 1.1-1.3 of Fan \cite{Fan-2007} we have the following results.

\begin{theorem}\label{C1,alpha.D}
	Let hypotheses H(D), H(p3) and H(A) be satisfied. Then, any weak solution $u$ of problem~\eqref{D} is of class $C_{\loc}^{1,\alpha}(\Omega)$ for some $\alpha \in (0,1)$. Furthermore, if in addition $\Gamma$ is of class $C^{1,\beta}$, then $u\in C^{1,\alpha}(\close)$ for some $\alpha \in (0,1)$.
\end{theorem}

\begin{theorem}\label{C1,alpha.N}
	Assume that $\Gamma$ is of class $C^{1,\beta}$ and let the hypotheses H(N) and H(A) be satisfied. Further, let  $p\in C_+(\close)\cap W^{1,\gamma}(\Omega)$ for some $\gamma>N$ and let $\mathcal{C}\in C(\Gamma\times\R,\R)$ be such that
	\begin{equation*}
		|\mathcal{C}(x_1,s_1)-\mathcal{C}(x_2,s_2)\leq \Lambda (\max\{|s_1|,|s_2|\})\left(|x_1-x_2|^{\delta_1}+|s_1-s_2|^{\delta_2}\right),
	\end{equation*}
	 for all $x_1,x_2\in\Gamma$, for all $s_1,s_2\in\R$, where $\Lambda$ is as in H(A). Then, any weak solution $u$ of problem \eqref{N} belongs to $C^{1,\alpha}(\close)$ for some $\alpha \in (0,1)$.
\end{theorem}

\begin{remark}
	Note that condition H(p3) is automatically satisfied in Theorem \ref{C1,alpha.N} due to $p\in C_+(\close)\cap W^{1,\gamma}(\Omega)$ for some $\gamma>N$ and \eqref{inclusions}.
\end{remark}
\section*{Acknowledgment}

K. Ho was supported by University of Economics Ho Chi Minh City, Vietnam. Y.-H. Kim was supported by the National Research Foundation of Korea (NRF) grant funded by the Korea government (MSIT) (NRF-2019R1F1A1057775). C. Zhang was supported by the National Natural Science Foundation of China (No. 12071098).


\end{document}